\documentclass[10pt]{article}

\usepackage{tikz, float}
\usepackage{amssymb}
\usepackage{url}
\usepackage{amscd}
\usepackage{amsthm}
\usepackage{amsmath}
\usepackage{graphicx}

\newtheorem {Theorem}  {Theorem}

\newtheorem {Conjecture} {Conjecture}

\newtheorem {Problem} {Problem}

\begin{document}
\baselineskip = 15pt
\bibliographystyle{plain}

\title{A proof of Ollinger's conjecture: undecidability of tiling the plane\\ with a set of $8$ polyominoes}
\date{}
\author{Chao Yang\\ 
              School of Mathematics and Statistics\\
              Guangdong University of Foreign Studies, Guangzhou, 510006, China\\
              sokoban2007@163.com, yangchao@gdufs.edu.cn\\
              \\
        Zhujun Zhang\\
              Government Data Management Center of\\
              Fengxian District, Shanghai, 201499, China\\
              zhangzhujun1988@163.com
              }
\maketitle

\begin{abstract}
    We give a proof of Ollinger's conjecture that the problem of tiling the plane with translated copies of a set of $8$ polyominoes is undecidable. The techniques employed in our proof include a different orientation for simulating the Wang tiles in polyomino and a new method for encoding the colors of Wang tiles.
\end{abstract}

\noindent{\textbf{Keywords}}:
tiling, Wang tiles, polyomino, undecidability\\
MSC2020: 52C20, 68Q17

\section{Introduction} 

Aperiodicity and undecidability are two important and related phenomena regarding plane tiling. One of the earliest studies on aperiodicity and undecidability of tiling dates back to Wang tiles \cite{wang61} introduced by Hao Wang in 1961. A \textit{Wang tile} is a unit square with each edge assigned a color. Given a finite set of Wang tiles (see Figure \ref{fig_wang_set} for an example), Wang considered the problem of tiling the entire plane with translated copies from the set, under the conditions that the tiles must be edge-to-edge and the color of common edges of any two adjacent Wang tiles must be the same. This is known as \textit{Wang's domino problem}. Wang once conjectured that if a set of Wang tiles can tile the plane, then it can tile the plane periodically. If Wang's conjecture is true, then his domino problem is decidable.

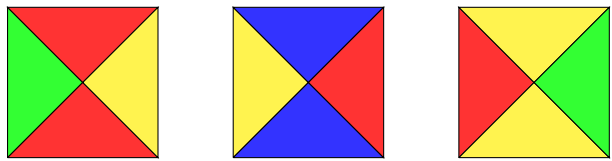
\begin{figure}[H]
\begin{center}
\begin{tikzpicture}

\draw [fill=green!80] (0,0)--(1,1)--(0,2)--(0,0);
\draw [fill=red!80] (0,0)--(2,0)--(0,2)--(2,2)--(0,0);
\draw [fill=yellow!80] (1,1)--(2,2)--(2,0)--(1,1);

\draw [fill=yellow!80] (3+0,0)--(3+1,1)--(3+0,2)--(3+0,0);
\draw [fill=blue!80] (3+0,0)--(3+2,0)--(3+0,2)--(3+2,2)--(3+0,0);
\draw [fill=red!80] (3+1,1)--(3+2,2)--(3+2,0)--(3+1,1);

\draw [fill=red!80] (6+0,0)--(6+1,1)--(6+0,2)--(6+0,0);
\draw [fill=yellow!80] (6+0,0)--(6+2,0)--(6+0,2)--(6+2,2)--(6+0,0);
\draw [fill=green!80] (6+1,1)--(6+2,2)--(6+2,0)--(6+1,1);

\end{tikzpicture}
\end{center}
\caption{A set of $3$ Wang tiles.}\label{fig_wang_set}
\end{figure}

Berge first found a set of Wang tiles that tiles the plane but only tiles the plane non-periodically. Such a set is called \textit{aperiodic}. By combining the facts of the existence of an aperiodic set of Wang tiles and the ability to simulate the Turing machine with Wang tiles, Berge showed that Wang's domino problem is undecidable.

\begin{Theorem}[\cite{b66}]
    Wang's domino problem is undecidable.
\end{Theorem}

Berge's first aperiodic set of Wang tiles has more than $20000$ tiles. Smaller and smaller aperiodic sets of Wang tiles are found in the following decades \cite{c96,kari96,r71}. Finally, with the help of computers, Jeandel and Rao showed that there exists an aperiodic set of $11$ Wang tiles and any sets of $10$ or less Wang tiles are periodic \cite{jr21}.

The plane tiling problems can be considered in different settings, by allowing tiles of other shapes, different matching conditions, and more transformation types (rotation and reflection in addition to translation). Under these more general settings, researchers find aperiodic sets of even fewer number of tiles. Notably, Penrose first found an aperiodic set of two tiles \cite{p79}. More strikingly, a series of aperiodic monotile has been discovered by Smith, Myers, Kaplan, and Goodman-Strauss \cite{smith23a,smith23b}, which solves the long-standing einstein problem. For more aperiodic tiling sets and their applications in quasicrystal, we refer to the books \cite{bg13, gs16}.

\begin{figure}[h]
\begin{center}
\begin{tikzpicture}[scale=0.6]

\draw [ fill=red!80] (0,0)--(1,1)--(0,2)--(0,0);
\draw [ fill=yellow!80] (0,0)--(1,1)--(2,0)--(0,0);
\draw [ fill=blue!80] (2,0)--(1,1)--(2,2)--(2,0);
\draw [ fill=green!80] (0,2)--(1,1)--(2,2)--(0,2);
 
\draw [ fill=blue!80] (3,0)--(4,1)--(3,2)--(3,0);
\draw [ fill=green!80] (3,0)--(4,1)--(5,0)--(3,0);
\draw [ fill=orange!80] (5,0)--(4,1)--(5,2)--(5,0);
\draw [ fill=yellow!80] (3,2)--(4,1)--(5,2)--(3,2);

\draw [thick](8,0)--(11,0)--(11,-1)--(11.5,-1)--(11.5,0)--(12,0)--(12,1.5)--(13,1.5)--(13,3)--(12.5,3)--(12.5,2)--(12,2)--(12,4)--(9.5,4)--(9.5,3.5)--(8.5,3.5)--(8.5,4)--(8,4)--(8,1.5)--(9,1.5)--(9,0.5)--(8.5,0.5)--(8.5,1)--(8,1)--(8,0);

\draw [thick](14,0)--(14.5,0)--(14.5,-0.5)--(15.5,-0.5)--(15.5,0)--(18,0)--(18,0.5)--(18.5,0.5)--(18.5,1)--(18,1)--(18,4)--(17.5,4)--(17.5,3)--(17,3)--(17,4)--(14,4)--(14,2)--(14.5,2)--(14.5,3)--(15,3)--(15,1.5)--(14,1.5)--(14,0);

\end{tikzpicture}
\end{center}
\caption{Wang tiles simulated by polyominoes.} \label{fig_golomb}
\end{figure}
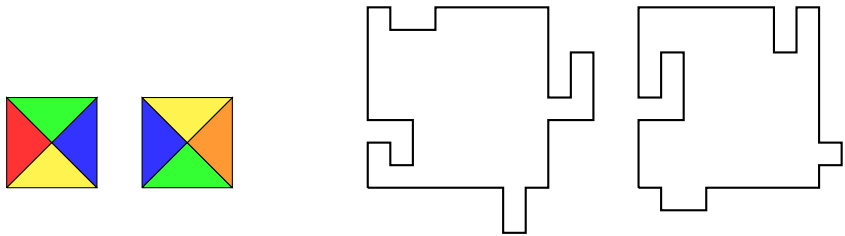

In this paper, we focus on translational tiling problems of the plane where only translated copies can be used and the tiles are matched by shapes only. In particular, we study the plane tiling problem with translated copies of a finite set of polyominoes. It has been shown by Golomb \cite{g70} that Wang's domino problem can be reduced to the problem of tiling with a set of polyominoes. Golomb's reduction method can be illustrated in Figure \ref{fig_golomb}, where a set of Wang tiles is emulated by a set of polyominoes with the same number. The color of Wang tiles can be simulated by bumps and dents added to each side of a large square polyomino. Therefore, the problem of tiling with a set of polyominoes is undecidable in general. Ollinger considered the following problem in which the number of polyominoes is fixed.

\begin{Problem}[$k$-Polyomino Tiling Problem] 
For a fixed positive integer $k$, is there an algorithm to decide whether a set of $k$ polyominoes can tile the plane by translated copies from the set?    
\end{Problem}

For $k=1$, it is known that if a single polyomino tiles the plane, then it can tile the plane periodically \cite{bn91}. So the $1$-polyomnino tiling problem is decidable, and there is a fast algorithm to decide whether a polyomino tiles the plane \cite{w15}. On the other hand, Ollinger initiated the study on the undecidability of a set of a fixed number of polyominoes and proved the following result.

\begin{Theorem}[\cite{o09}]
    The $11$-polyomino tiling problem is undecidable. 
\end{Theorem}

Based on the existence of an aperiodic set of $8$ polyominoes \cite{ags92}, Ollinger made the following conjecture \cite{o09}.

\begin{Conjecture}[Ollinger's conjecture\cite{o09}]
     The $k$-polyomino tiling problem is undecidable, for $k=8,9,10$.
\end{Conjecture}

Obviously, to prove Ollinger's conjecture, it suffices to prove the case $k=8$. The previous work of the first author settled the conjecture for the $k=9,10$ cases \cite{yang23,yang23b}. This paper resolves Ollinger's conjecture completely by proving the following theorem.

\begin{Theorem}\label{thm_main}
   The $8$-polyomino tiling problem is undecidable. 
\end{Theorem}

It is worth mentioning that Greenfeld and Tao show that translational tiling problems can be undecidable for one or two tiles in certain spaces other than the plane \cite{gt23}. See Table 1 of \cite{gt23} for a relatively comprehensive summary of previous results on the numbers of tiles needed to obtain aperiodicity or undecidability for translational tilings. For the plane, our result is the current record holder of the smallest size of set of tiles whose corresponding translational tiling problem is undecidable.
    
To prove Theorem \ref{thm_main}, we incorporate several new techniques in reducing Wang's domino problem to the $k$-polyomino tiling problem. In the proof of Ollinger's original result on $11$ polyominoes as well as the proof of improvements to $10$ or $9$ polyominoes by the first author, there is an angle of $\pi/4$ between the orientations of the simulated Wang tiles and the polyominoes. In this paper, we change the way of simulating Wang tiles such that the orientations of the simulated Wang tiles and the polyominoes are the same, namely they both align to the same set of horizontal and vertical lines of the plane. Besides the difference in orientation of the simulated Wang tiles, we also introduce new methods for encoding the colors of the Wang tiles in our proof.

The rest of the paper is organized as follows. Section \ref{sec_main} proves our main result. Section \ref{sec_con} concludes with remarks on future work.

\section{Proof of Theorem \ref{thm_main}}\label{sec_main}

\begin{proof}[Proof of Theorem \ref{thm_main}] 
We prove the theorem by reduction from Wang's domino problem. We follow the framework of Ollinger \cite{o09} in the reduction. One of the key ideas in Ollinger's framework is to encode an arbitrary set of Wang tiles in one polyomino, which is called the meat polyomino. As a result, a set of an arbitrary number of Wang tiles is simulated by a fixed number of polyominoes. In Ollinger's original proof, a total of $11$ polyominoes are needed. They are a meat, a jaw, a filler, $4$ teeth, and $4$ links. We follow Ollinger's convention in naming the polyominoes, except that links are called wires by Ollinger \cite{o09}. With some innovation built on Ollinger's framework, we will show that $8$ polyominoes suffice to do this job. The $8$ polyominoes in our construction are a meat, a jaw, a filler, $3$ teeth, and $2$ links, with a decrease in the number for both teeth and links.

 Except the teeth polyominoes, all other $5$ polyominoes in our set can be built upon blown-up polyominoes by a factor of $9$, which we will refer to as \textit{base} polyominoes. In other words, the base polyomino is built by usually a lot of $9\times 9$ square polyominoes. Then bumps and dents are added to the sides of some of the $9\times 9$ squares to form our final meat, jaw, filler, and link polyominoes. The teeth polyominoes are tiny with a size smaller than a $9\times 9$ square, and they are used to fill the gaps between the dents of other polyominoes in the tiling. We take the set of $3$ Wang tiles illustrated in Figure \ref{fig_wang_set} as an example to show how the $8$ polyominoes are constructed. We will give complete and detailed illustrations of all the $8$ polyominoes corresponding to this set of $3$ Wang tiles. For a general set of $n$ Wang tiles with a total number of $m$ different colors, we will use boundary words to describe the base polyominoes.

 The single \textit{meat} polyomino to simulate the set of $3$ Wang tiles of Figure \ref{fig_wang_set} is illustrated in Figure \ref{fig_meat}. Note that there are $4$ different colors in this set of Wang tiles, so $2$ bits is sufficient to encode the colors. In general, if a set of Wang tiles has a total number of $m$ different colors, we need $t=\lceil \log_2 m \rceil$ bits for encoding the colors. Note also that each small gray square in Figure \ref{fig_meat} represents a $9\times 9$ square (the same applies to Figure \ref{fig_jaw}, Figure \ref{fig_jaw_zoom_in} and Figure \ref{fig_filler}), so the boundary word for the base polyomino of this meat polyomino is (starting from the upper left corner and going counterclockwise)
 $$d^9 \Bigl(d^9r^9(d^9)^3(r^9)^2(d^9)^2(r^9)^3\Bigr)^3 u^9 r^9 \Bigl((u^9)^3 (l^9)^2(u^9)^2 (l^9)^3 u^9l^9\Bigr)^3 l^9.$$

There are $T$-shape bumps added to the base polyomino, in different orientations. For example, a $T$-shape bump on a north side (see the orange line segments in Figure \ref{fig_meat}) added to the base polyomino can be described by the boundary word $l^4ur(ul)^2d^3l^4$ (replacing the original side $l^9$ of length $9$). The $T$-shape bumps are to be matched with the $T$-shape dents in the jaw polyomino which will be introduced later.

Besides the bumps, there are two kinds of dents (up to different orientations) added to the base polyomino. One kind of dents is $1$ unit deeper than the other, so we refer to them as a \textit{normal dent} and a \textit{deeper dent}, respectively. On the north side, the boundary words for the dents are either $dlul d^2l^2drdl^3uru l^2u^2ldlu$ or $d^2lul d^2l^2drdl^3uru l^2u^2ldlu^2$, where the latter one is deeper. The dents in the meat polyomino are used to encode the colors of Wang tiles. As we have mentioned earlier, for the set of Wang tiles in Figure \ref{fig_wang_set}, we have to put two dents side by side in Figure \ref{fig_meat} to encode the $4$ colors. The $3$ Wang tiles are simulated in sequence in the meat, from northwest to southeast, with the first Wang tile simulated inside the area enclosed by the dashed lines illustrated in Figure \ref{fig_meat}. The segments of the boundary of the base polyomino of the meat before the normal or deeper dents are added are shown in the color which the dents are encoding.

Next, we explain in more detail how the colors are encoded. Suppose that the $4$ colors red, green, blue, and yellow are encoded by $00$, $01$, $10$, and $11$, respectively. We read the codes encoded by the dents of the meat from left to right, or from top to bottom, depending on whether the dents are on horizontal sides or vertical sides. A normal dent encodes $0$ when it appears on the north side or the west side, and encodes $1$ when on the south or east sides. A deeper dent encodes $1$ when on the north or west sides, and encodes $0$ when on the south or east sides. The reason for this way of encoding is that we will match the dents of the north side and the south (or west and east) by the link polyominoes introduced later. The link polyomino can only match a normal dent with a deeper dent.


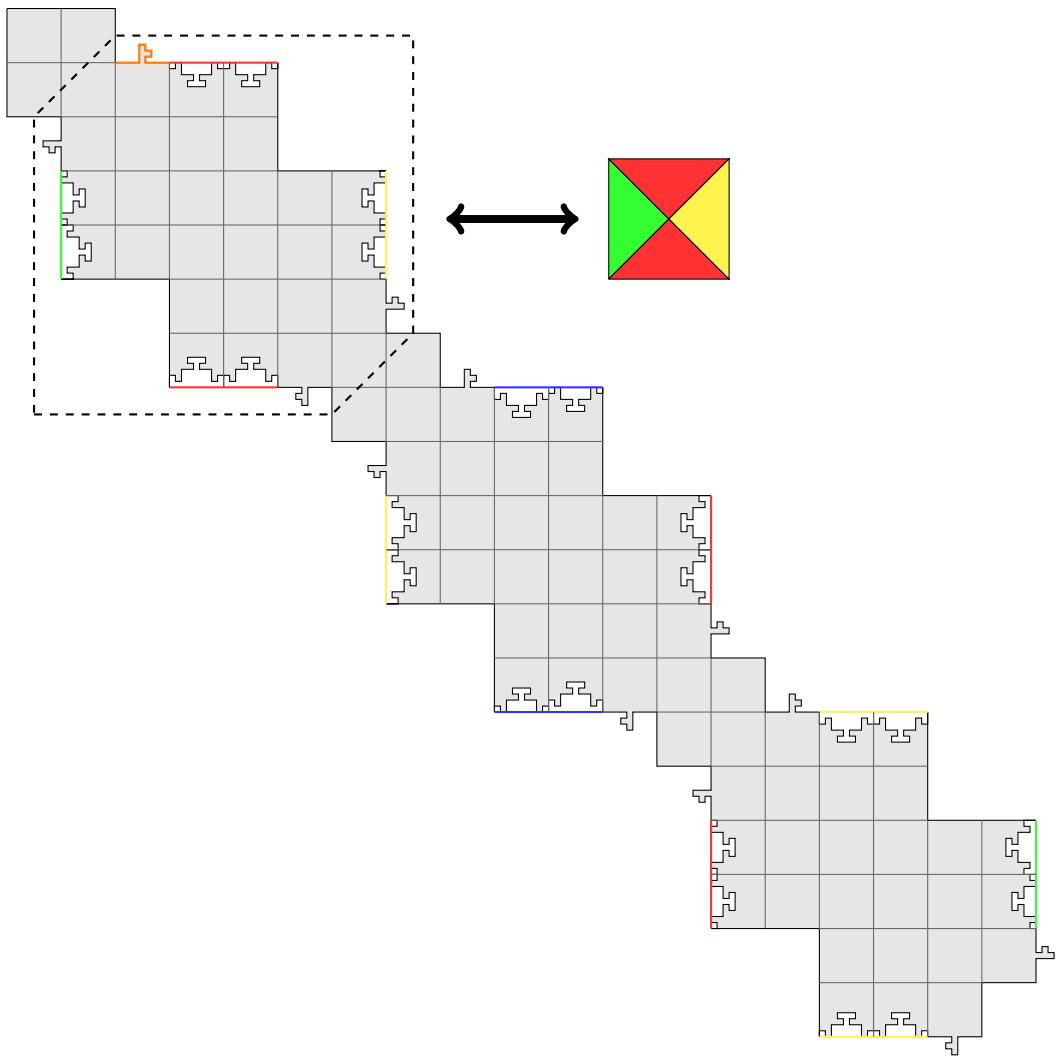
\begin{figure}[H]
\begin{center}
\begin{tikzpicture}[scale=0.08]

\draw [fill=gray!20] (0,63)--(18,63)
--(18,54)--(27-5,66-12)--(27-5,69-12)--(28-5,69-12)--(28-5,68-12)--(29-5,68-12)--(29-5,67-12)--(28-5,67-12)--(28-5,66-12)--(27,54)
--(34-7,66-12)--(34-7,65-12)--(35-7,65-12)--(35-7,66-12)--(36-7,66-12)--(36-7,64-12)--(38-7,64-12)--(38-7,63-12)--(37-7,63-12)--(37-7,62-12)--(40-7,62-12)--(40-7,63-12)--(39-7,63-12)--(39-7,64-12)--(41-7,64-12)--(41-7,66-12)--(42-7,66-12)--(42-7,65-12)--(43-7,65-12)
--(45-9,65-12)--(46-9,65-12)--(46-9,66-12)--(47-9,66-12)--(47-9,64-12)--(49-9,64-12)--(49-9,63-12)--(48-9,63-12)--(48-9,62-12)--(51-9,62-12)--(51-9,63-12)--(50-9,63-12)--(50-9,64-12)--(52-9,64-12)--(52-9,66-12)--(53-9,66-12)--(53-9,65-12)--(54-9,65-12)--(54-9,66-12)
--(45,54)--(45,36)--(63,36)
--(77-14,43-7)--(1+75-14,43-7)--(1+75-14,42-7)--(1+76-14,42-7)--(1+76-14,41-7)--(1+74-14,41-7)--(1+74-14,39-7)--(1+73-14,39-7)--(1+73-14,40-7)--(1+72-14,40-7)--(1+72-14,37-7)--(1+73-14,37-7)--(1+73-14,38-7)--(1+74-14,38-7)--(1+74-14,36-7)--(1+76-14,36-7)--(1+76-14,35-7)--(1+75-14,35-7)--(1+75-14,34-7)--(77-14,34-7)
--(77-14,32-5)--(1+75-14,32-5)--(1+75-14,31-5)--(1+76-14,31-5)--(1+76-14,30-5)--(1+74-14,30-5)--(1+74-14,28-5)--(1+73-14,28-5)--(1+73-14,29-5)--(1+72-14,29-5)--(1+72-14,26-5)--(1+73-14,26-5)--(1+73-14,27-5)--(1+74-14,27-5)--(1+74-14,25-5)--(1+76-14,25-5)--(1+76-14,24-5)--(1+75-14,24-5)--(1+75-14,23-5)--(77-14,23-5)
--(63,18)
--(77-14,17-3)--(78-14,17-3)--(78-14,18-3)--(79-14,18-3)--(79-14,17-3)--(80-14,17-3)--(80-14,16-3)--(77-14,16-3)
--(63,9)--(72,9)
--(72,0)
--(93-17,0)--(93-17,3)--(94-17,3)--(94-17,2)--(95-17,2)--(95-17,1)--(94-17,1)--(94-17,0)
--(81,0)
--(100-19,0)--(100-19,-2)--(101-19,-2)--(101-19,-1)--(102-19,-1)--(102-19,-3)--(104-19,-3)--(104-19,-4)--(103-19,-4)--(103-19,-5)--(106-19,-5)--(106-19,-4)--(105-19,-4)--(105-19,-3)--(107-19,-3)--(107-19,-1)--(108-19,-1)--(108-19,-2)--(109-19,-2)--(109-19,0)
--(111-21,0)--(111-21,-1)--(112-21,-1)--(112-21,0)--(113-21,0)--(113-21,-2)--(115-21,-2)--(115-21,-3)--(114-21,-3)--(114-21,-4)--(117-21,-4)--(117-21,-3)--(116-21,-3)--(116-21,-2)--(118-21,-2)--(118-21,0)--(119-21,0)--(119-21,-1)--(120-21,-1)--(120-21,0)
--(99,0)--(99,-18)--(117,-18)
--(40+77,43-61)--(40+75,43-61)--(40+75,42-61)--(40+76,42-61)--(40+76,41-61)--(40+74,41-61)--(40+74,39-61)--(40+73,39-61)--(40+73,40-61)--(40+72,40-61)--(40+72,37-61)--(40+73,37-61)--(40+73,38-61)--(40+74,38-61)--(40+74,36-61)--(40+76,36-61)--(40+76,35-61)--(40+75,35-61)--(40+75,34-61)--(40+77,34-61)
--(40+77,32-59)--(40+75,32-59)--(40+75,31-59)--(40+76,31-59)--(40+76,30-59)--(40+74,30-59)--(40+74,28-59)--(40+73,28-59)--(40+73,29-59)--(40+72,29-59)--(40+72,26-59)--(40+73,26-59)--(40+73,27-59)--(40+74,27-59)--(40+74,25-59)--(40+76,25-59)--(40+76,24-59)--(40+75,24-59)--(40+75,23-59)--(40+77,23-59)
--(117,-36)
--(40+77,17-57)--(40+78,17-57)--(40+78,18-57)--(40+79,18-57)--(40+79,17-57)--(40+80,17-57)--(40+80,16-57)--(40+77,16-57)
--(117,-45)
--(126,-45)--(126,-54)
--(37+93,0-54)--(37+93,3-54)--(37+94,3-54)--(37+94,2-54)--(37+95,2-54)--(37+95,1-54)--(37+94,1-54)--(37+94,0-54)
--(135,-54)
--(101+34,66-120)--(101+34,65-121)--(101+35,65-121)--(101+35,66-121)--(101+36,66-121)--(101+36,64-121)--(101+38,64-121)--(101+38,63-121)--(101+37,63-121)--(101+37,62-121)--(101+40,62-121)--(101+40,63-121)--(101+39,63-121)--(101+39,64-121)--(101+41,64-121)--(101+41,66-121)--(101+42,66-121)--(101+42,65-121)--(101+43,65-121)--(101+43,66-120)
--(99+45,66-120)--(99+45,65-121)--(99+46,65-121)--(99+46,66-121)--(99+47,66-121)--(99+47,64-121)--(99+49,64-121)--(99+49,63-121)--(99+48,63-121)--(99+48,62-121)--(99+51,62-121)--(99+51,63-121)--(99+50,63-121)--(99+50,64-121)--(99+52,64-121)--(99+52,66-121)--(99+53,66-121)--(99+53,65-121)--(99+54,65-121)--(99+54,66-120)
--(153,-54)
--(153,-72)--(171,-72)
--(94+77,43-115)--(94+75,43-115)--(94+75,42-115)--(94+76,42-115)--(94+76,41-115)--(94+74,41-115)--(94+74,39-115)--(94+73,39-115)--(94+73,40-115)--(94+72,40-115)--(94+72,37-115)--(94+73,37-115)--(94+73,38-115)--(94+74,38-115)--(94+74,36-115)--(94+76,36-115)--(94+76,35-115)--(94+75,35-115)--(94+75,34-115)--(94+77,34-115)
--(94+77,32-113)--(95+75,32-113)--(95+75,31-113)--(95+76,31-113)--(95+76,30-113)--(95+74,30-113)--(95+74,28-113)--(95+73,28-113)--(95+73,29-113)--(95+72,29-113)--(95+72,26-113)--(95+73,26-113)--(95+73,27-113)--(95+74,27-113)--(95+74,25-113)--(95+76,25-113)--(95+76,24-113)--(95+75,24-113)--(95+75,23-113)--(94+77,23-113)
--(171,-90)
--(94+77,17-111)--(94+78,17-111)--(94+78,18-111)--(94+79,18-111)--(94+79,17-111)--(94+80,17-111)--(94+80,16-111)--(94+77,16-111)
--(171,-99)--(162,-99)--(162,-108)
--(97+61,0-108)--(97+61,-3-108)--(97+60,-3-108)--(97+60,-2-108)--(97+59,-2-108)--(97+59,-1-108)--(97+60,-1-108)--(97+60,0-108)
--(153,-108)
--(99+54,0-108)--(99+54,1-108)--(99+53,1-108)--(99+53,0-108)--(99+52,0-108)--(99+52,2-108)--(99+50,2-108)--(99+50,3-108)--(99+51,3-108)--(99+51,4-108)--(99+48,4-108)--(99+48,3-108)--(99+49,3-108)--(99+49,2-108)--(99+47,2-108)--(99+47,0-108)--(99+46,0-108)--(99+46,1-108)--(99+45,1-108)--(99+45,0-108)
--(101+43,0-108)--(101+43,1-108)--(101+42,1-108)--(101+42,0-108)--(101+41,0-108)--(101+41,2-108)--(101+39,2-108)--(101+39,3-108)--(101+40,3-108)--(101+40,4-108)--(101+37,4-108)--(101+37,3-108)--(101+38,3-108)--(101+38,2-108)--(101+36,2-108)--(101+36,0-108)--(101+35,0-108)--(101+35,1-108)--(101+34,1-108)--(101+34,0-108)
--(135,-108)--(135,-90)--(117,-90)
--(106+11,23-113)--(105+13,23-113)--(105+13,24-113)--(105+12,24-113)--(105+12,25-113)--(105+14,25-113)--(105+14,27-113)--(105+15,27-113)--(105+15,26-113)--(105+16,26-113)--(105+16,29-113)--(105+15,29-113)--(105+15,28-113)--(105+14,28-113)--(105+14,30-113)--(105+12,30-113)--(105+12,31-113)--(105+13,31-113)--(105+13,32-113)--(106+11,32-113)
--(106+11,34-115)--(106+12,34-115)--(106+12,35-115)--(106+11,35-115)--(106+11,36-115)--(106+13,36-115)--(106+13,38-115)--(106+14,38-115)--(106+14,37-115)--(106+15,37-115)--(106+15,40-115)--(106+14,40-115)--(106+14,39-115)--(106+13,39-115)--(106+13,41-115)--(106+11,41-115)--(106+11,42-115)--(106+12,42-115)--(106+12,43-115)--(106+11,43-115)
--(117,-72)
--(106+11,49-117)--(106+10,49-117)--(106+10,48-117)--(106+9,48-117)--(106+9,49-117)--(106+8,49-117)--(106+8,50-117)--(106+11,50-117)
--(117,-63)--(108,-63)--(108,-54)
--(43+61,0-54)--(43+61,-3-54)--(43+60,-3-54)--(43+60,-2-54)--(43+59,-2-54)--(43+59,-1-54)--(43+60,-1-54)--(43+60,0-54)
--(99,-54)
--(45+54,0-54)--(45+54,1-53)--(45+53,1-53)--(45+53,0-53)--(45+52,0-53)--(45+52,2-53)--(45+50,2-53)--(45+50,3-53)--(45+51,3-53)--(45+51,4-53)--(45+48,4-53)--(45+48,3-53)--(45+49,3-53)--(45+49,2-53)--(45+47,2-53)--(45+47,0-53)--(45+46,0-53)--(45+46,1-53)--(45+45,1-53)--(45+45,0-54)
--(47+43,0-54)--(47+43,1-54)--(47+42,1-54)--(47+42,0-54)--(47+41,0-54)--(47+41,2-54)--(47+39,2-54)--(47+39,3-54)--(47+40,3-54)--(47+40,4-54)--(47+37,4-54)--(47+37,3-54)--(47+38,3-54)--(47+38,2-54)--(47+36,2-54)--(47+36,0-54)--(47+35,0-54)--(47+35,1-54)--(47+34,1-54)--(47+34,0-54)
--(81,-54)--(81,-36)--(63,-36)
--(52+11,23-59)--(52+13,23-59)--(52+13,24-59)--(52+12,24-59)--(52+12,25-59)--(52+14,25-59)--(52+14,27-59)--(52+15,27-59)--(52+15,26-59)--(52+16,26-59)--(52+16,29-59)--(52+15,29-59)--(52+15,28-59)--(52+14,28-59)--(52+14,30-59)--(52+12,30-59)--(52+12,31-59)--(52+13,31-59)--(52+13,32-59)--(52+11,32-59)
--(52+11,34-61)--(53+12,34-61)--(53+12,35-61)--(53+11,35-61)--(53+11,36-61)--(53+13,36-61)--(53+13,38-61)--(53+14,38-61)--(53+14,37-61)--(53+15,37-61)--(53+15,40-61)--(53+14,40-61)--(53+14,39-61)--(53+13,39-61)--(53+13,41-61)--(53+11,41-61)--(53+11,42-61)--(53+12,42-61)--(53+12,43-61)--(52+11,43-61)
--(63,-18)
--(52+11,49-63)--(52+10,49-63)--(52+10,48-63)--(52+9,48-63)--(52+9,49-63)--(52+8,49-63)--(52+8,50-63)--(52+11,50-63)
--(63,-9)
--(54,-9)--(54,0)
--(50,0)--(50,-3)--(49,-3)--(49,-2)--(48,-2)--(48,-1)--(49,-1)--(49,0)
--(45,0)
--(54-9,0)--(54-9,1+1)--(53-9,1+1)--(53-9,0+1)--(52-9,0+1)--(52-9,2+1)--(50-9,2+1)--(50-9,3+1)--(51-9,3+1)--(51-9,4+1)--(48-9,4+1)--(48-9,3+1)--(49-9,3+1)--(49-9,2+1)--(47-9,2+1)--(47-9,0+1)--(46-9,0+1)--(46-9,1+1)--(45-9,1+1)--(45-9,0)
--(43-7,0)--(43-7,1+1)--(42-7,1+1)--(42-7,0+1)--(41-7,0+1)--(41-7,2+1)--(39-7,2+1)--(39-7,3+1)--(40-7,3+1)--(40-7,4+1)--(37-7,4+1)--(37-7,3+1)--(38-7,3+1)--(38-7,2+1)--(36-7,2+1)--(36-7,0+1)--(35-7,0+1)--(35-7,1+1)--(34-7,1+1)--(34-7,0)
--(27,0)
--(27,18)--(9,18)
--(11-2,23-5)--(13-2,23-5)--(13-2,24-5)--(12-2,24-5)--(12-2,25-5)--(14-2,25-5)--(14-2,27-5)--(15-2,27-5)--(15-2,26-5)--(16-2,26-5)--(16-2,29-5)--(15-2,29-5)--(15-2,28-5)--(14-2,28-5)--(14-2,30-5)--(12-2,30-5)--(12-2,31-5)--(13-2,31-5)--(13-2,32-5)--(11-2,32-5)
--(11-2,34-7)--(12-2,34-7)--(12-2,35-7)--(11-2,35-7)--(11-2,36-7)--(13-2,36-7)--(13-2,38-7)--(14-2,38-7)--(14-2,37-7)--(15-2,37-7)--(15-2,40-7)--(14-2,40-7)--(14-2,39-7)--(13-2,39-7)--(13-2,41-7)--(11-2,41-7)--(11-2,42-7)--(12-2,42-7)--(12-2,43-7)--(11-2,43-7)
--(9,36)
--(11-2,49-9)--(10-2,49-9)--(10-2,48-9)--(9-2,48-9)--(9-2,49-9)--(8-2,49-9)--(8-2,50-9)--(11-2,50-9)
--(9,45)--(0,45)--(0,63);

\draw [color=red!80, thick] (27,54)--(45,54);
\draw [color=red!80, thick] (27,0)--(45,0);
\draw [color=yellow!80, thick] (63,36)--(63,18);
\draw [color=green!80, thick] (9,36)--(9,18);

\draw [color=blue!80, thick] (81,0)--(99,0);
\draw [color=blue!80, thick] (81,-54)--(99,-54);
\draw [color=red!80, thick] (117,-18)--(117,-36);
\draw [color=yellow!80, thick] (63,-18)--(63,-36);

\draw [color=yellow!80, thick] (135,-54)--(153,-54);
\draw [color=yellow!80, thick] (135,-108)--(153,-108);
\draw [color=green!80, thick] (171,-72)--(171,-90);
\draw [color=red!80, thick] (117,-72)--(117,-90);

\draw [color=black!60] (0,54)--(18,54)--(18,18);
\draw [color=black!60] (9,63)--(9,45)--(45,45);
\draw [color=black!60] (9,36)--(45,36)--(45,0);
\draw [color=black!60] (9,27)--(63,27);
\draw [color=black!60] (27,54)--(27,18)--(63,18);
\draw [color=black!60] (36,54)--(36,0);
\draw [color=black!60] (27,9)--(63,9)--(63,-9)--(99,-9);
\draw [color=black!60] (54,36)--(54,0)--(72,0)--(72,-36);
\draw [color=black!60] (81,0)--(81,-36)--(117,-36);
\draw [color=black!60] (63,-18)--(99,-18)--(99,-54);
\draw [color=black!60] (63,-27)--(117,-27);
\draw [color=black!60] (90,0)--(90,-54);
\draw [color=black!60] (81,-45)--(117,-45)--(117,-63)--(153,-63);
\draw [color=black!60] (108,-18)--(108,-54)--(126,-54)--(126,-90);
\draw [color=black!60] (117,-72)--(153,-72)--(153,-108);
\draw [color=black!60] (117,-81)--(171,-81);
\draw [color=black!60] (135,-54)--(135,-90)--(171,-90);
\draw [color=black!60] (144,-54)--(144,-108);
\draw [color=black!60] (135,-99)--(162,-99)--(162,-72);

\draw [dashed, thick] (4.5,-4.5)--(54,-4.5)--(67.5,9)--(67.5,58.5)--(18,58.5)--(4.5,45)--(4.5,-4.5);

\draw [orange, thick] (18,54)--(27-5,66-12)--(27-5,69-12)--(28-5,69-12)--(28-5,68-12)--(29-5,68-12)--(29-5,67-12)--(28-5,67-12)--(28-5,66-12)--(27,54);

\foreach \x in {100}
\foreach \y in {18}
{
\draw [fill=green!80] (\x+0,0+\y)--(\x+10,10+\y)--(\x+0,20+\y)--(\x+0,0+\y);
\draw [fill=red!80] (\x+0,0+\y)--(\x+20,0+\y)--(\x+0,20+\y)--(\x+20,20+\y)--(\x+0,0+\y);
\draw [fill=yellow!80] (\x+10,10+\y)--(\x+20,20+\y)--(\x+20,0+\y)--(\x+10,10+\y);
}
\draw  [<->, line width=3,  black] (73,28)--(95,28);

\end{tikzpicture}
\end{center}
\caption{The meat.}\label{fig_meat}
\end{figure}


In general, for a set of $n$ Wang tiles with $m$ different colors ($t=\lceil \log_2 m \rceil$), we can construct the meat by adding bumps and dents in the same way to the following base polyomino
 $$d^9 \Bigl(d^9r^9(d^9)^{t+1}(r^9)^2(d^9)^2(r^9)^{t+1}\Bigr)^n u^9 r^9 \Bigl((u^9)^{t+1} (l^9)^2(u^9)^2 (l^9)^{t+1} u^9l^9\Bigr)^n l^9.$$

The \textit{jaw} polyomino (see Figure \ref{fig_jaw} and Figure \ref{fig_jaw_zoom_in}) is used to select one of the three Wang tiles simulated by the meat. Figure \ref{fig_jaw} illustrates the base polyomino of the jaw, where dents are added to the orange sides, and the areas inside the violet circles are identical (regarding the dents of the orange sides). A zoomed-in view inside a violet circle is illustrated in Figure \ref{fig_jaw_zoom_in}. The jaw is the biggest polyomino in our construction. In general, the boundary word of its base polyomino is (starting from the lower left corner)
\begin{equation*}
\begin{aligned}(r^9)^{2(n-1)(t+4)+3}u^9l^9 \Bigl ( u^9(l^9)^{t+1}(u^9)^2(l^9)^2(u^9)^{t+1}l^9\Bigr )^{n-1} u^{18}r^{18} \Bigl ( d^9(r^9)^{t+1}(d^9)^2(r^9)^2(d^9)^{t+1}r^9\Bigr )^{n-1} d^9r^9 (u^9)^{2(n-1)(t+4)+3}\\
(l^9)^{2(n-1)(t+4)+3}d^9\Bigl ( r^9d^9(r^9)^{t+1}(d^9)^2(r^9)^2(d^9)^{t+1}\Bigr )^{n-1} l^{9}d^{9} \Bigl ( (l^9)^{t+1}(u^9)^2(l^9)^2(u^9)^{t+1}l^9u^9\Bigr )^{n-1} l^9 (d^9)^{2(n-1)(t+4)+3}.
\end{aligned}
\end{equation*}
Figure \ref{fig_jaw} depicted this base polyomino for $n=3$ and $t=2$.


\begin{figure}[H]
\begin{center}
\begin{tikzpicture}[scale=0.5]

\draw [fill=gray!20] (0,0)--(27,0)
--(27,1)--(26,1)--(26,2)--(23,2)--(23,4)--(21,4)--(21,7)--(20,7)--(20,8)--(17,8)--(17,10)--(15,10)--(15,13)--(14,13)--(14,15)--(16,15)--(16,14)--(19,14)--(19,12)--(21,12)--(21,9)--(22,9)--(22,8)--(25,8)--(25,6)--(27,6)--(27,3)--(28,3)--(28,2)
--(29,2)
--(29,29)
--(2,29)
--(2,28)--(3,28)--(3,27)--(6,27)--(6,25)--(8,25)--(8,22)--(9,22)--(9,21)--(12,21)--(12,19)--(14,19)--(14,16)--(13,16)--(13,15)--(10,15)--(10,17)--(8,17)--(8,20)--(7,20)--(7,21)--(4,21)--(4,23)--(2,23)--(2,26)--(1,26)--(1,27)
--(0,27)--(0,0);

\foreach \x in {0,...,29}  
{ 
\draw [ color=black!60] (\x,0)--(\x,29);
\draw [ color=black!60] (0,\x)--(29,\x);
}

\draw[color=violet!60, very thick](5,24) circle (4);
\draw[color=violet!60, very thick](11,18) circle (4);
\draw[color=violet!60, very thick](18,11) circle (4);
\draw[color=violet!60, very thick](24,5) circle (4);
\draw[color=violet!60, very thick] (29,2.873) arc (104.48:165.52:4);
\draw[color=violet!60, very thick] (0,26.127) arc (284.48:345.52:4);
\draw[color=orange, very thick] (3,27)--(6,27);
\draw[color=orange, very thick] (4,21)--(7,21);
\draw[color=orange, very thick] (2,23)--(2,26);
\draw[color=orange, very thick] (8,22)--(8,25);
\draw[color=orange, very thick] (6+3,27-6)--(6+6,27-6);
\draw[color=orange, very thick] (6+4,21-6)--(6+7,21-6);
\draw[color=orange, very thick] (6+2,23-6)--(6+2,26-6);
\draw[color=orange, very thick] (6+8,22-6)--(6+8,25-6);
\draw[color=orange, very thick] (13+3,27-13)--(13+6,27-13);
\draw[color=orange, very thick] (13+4,21-13)--(13+7,21-13);
\draw[color=orange, very thick] (13+2,23-13)--(13+2,26-13);
\draw[color=orange, very thick] (13+8,22-13)--(13+8,25-13);
\draw[color=orange, very thick] (19+3,27-19)--(19+6,27-19);
\draw[color=orange, very thick] (19+4,21-19)--(19+7,21-19);
\draw[color=orange, very thick] (19+2,23-19)--(19+2,26-19);
\draw[color=orange, very thick] (19+8,22-19)--(19+8,25-19);

\draw[color=orange, very thick] (28,2)--(29,2);
\draw[color=orange, very thick] (27,1)--(27,0);

\draw[color=orange, very thick] (0,27)--(1,27);
\draw[color=orange, very thick] (2,28)--(2,29);

\end{tikzpicture}
\end{center}
\caption{The jaw.}\label{fig_jaw}
\end{figure}


Now we have a closer look at the dents of the orange sides of the jaw in Figure \ref{fig_jaw_zoom_in}. There are two kinds of dents here. One is the deeper dent the same as those of the meat which is used to encode the colors of Wang tiles there. The other is the $T$-shape dent which matches the $T$-shape bumps of the meat. The second kind of dent forces the meat to be put inside the jaw in one of several possible locations. 

One key feature of the jaw is that it can contain at most $2$ simulated Wang tiles of a meat (can contain at most $n-1$ simulated Wang tiles in the general case). In other words, at least one simulated Wang tile has to be left outside the jaw. On the other hand, the $T$-shape bumps of the meat make sure that at most one simulated Wang tile can be left outside the jaw. Every meat must be embraced by two jaw polyominoes from two directions, one from the northwest and the other from the southeast. But the meat has the freedom to choose which of the $3$ simulated Wang tiles to be left outside (see the overall tiling pattern illustrated later in Figure \ref{fig_pattern} for $3$ different ways by which a meat can be placed inside the jaws).

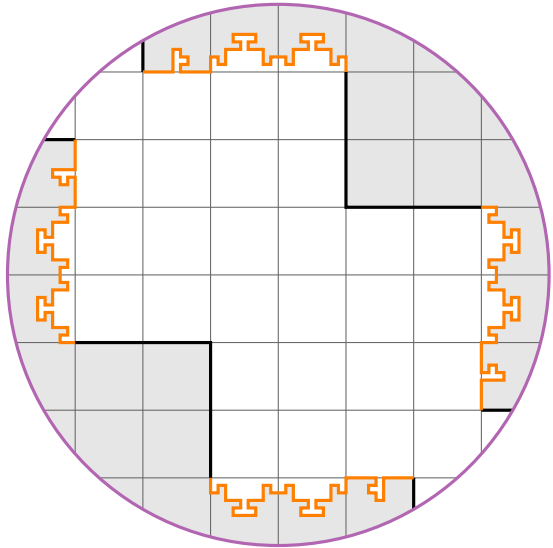
\begin{figure}[H]
\begin{center}
\begin{tikzpicture}

\begin{scope}
    \clip (0,0)--(5.4,0)--(5.4,0.9)
    --(6.1-1.1,1.1-0.2)--(6.1-1.1,0.8-0.2)--(6.0-1.1, 0.8-0.2)--(6.0-1.1,0.9-0.2)--(5.9-1.1,0.9-0.2)--(5.9-1.1,1.0-0.2)--(6.0-1.1,1.0-0.2)--(6.0-1.1,1.1-0.2)
    --(5.4-0.9,1.1-0.2)--(5.4-0.9,0.9-0.2)--(5.3-0.9,0.9-0.2)--(5.3-0.9,1.0-0.2)--(5.2-0.9,1.0-0.2)--(5.2-0.9, 0.8-0.2)--(5.0-0.9, 0.8-0.2)--(5.0-0.9, 0.7-0.2)--(5.1-0.9,0.7-0.2)--(5.1-0.9,0.6-0.2)--(4.8-0.9,0.6-0.2)--(4.8-0.9,0.7-0.2)--(4.9-0.9,0.7-0.2)--(4.9-0.9,0.8-0.2)--(4.7-0.9,0.8-0.2)--(4.7-0.9,1.0-0.2)--(4.6-0.9,1.0-0.2)--(4.6-0.9,0.9-0.2)--(4.5-0.9,0.9-0.2)--(4.5-0.9,1.1-0.2)
     --(5.4-1.8,1.1-0.2)--(5.4-1.8,0.9-0.2)--(5.3-1.8,0.9-0.2)--(5.3-1.8,1.0-0.2)--(5.2-1.8,1.0-0.2)--(5.2-1.8, 0.8-0.2)--(5.0-1.8, 0.8-0.2)--(5.0-1.8, 0.7-0.2)--(5.1-1.8,0.7-0.2)--(5.1-1.8,0.6-0.2)--(4.8-1.8,0.6-0.2)--(4.8-1.8,0.7-0.2)--(4.9-1.8,0.7-0.2)--(4.9-1.8,0.8-0.2)--(4.7-1.8,0.8-0.2)--(4.7-1.8,1.0-0.2)--(4.6-1.8, 1.0-0.2)--(4.6-1.8,0.9-0.2)--(4.5-1.8,0.9-0.2)--(4.5-1.8,1.1-0.2)
    --(2.7,0.9)--(2.7,2.7)--(0.9,2.7)
    --(1.1-0.2, 3.4-0.7)--(0.9-0.2, 3.4-0.7)-- (0.9-0.2, 3.5-0.7)--(1.0-0.2, 3.5-0.7)--(1.0-0.2, 3.6-0.7)--(0.8-0.2, 3.6-0.7)--(0.8-0.2, 3.8-0.7)--(0.7-0.2, 3.8-0.7)--(0.7-0.2, 3.7-0.7)--(0.6-0.2, 3.7-0.7)--(0.6-0.2, 4.0-0.7)--(0.7-0.2, 4.0-0.7)--(0.7-0.2, 3.9-0.7)--(0.8-0.2, 3.9-0.7)--(0.8-0.2, 4.1-0.7)--(1.0-0.2,4.1-0.7)--(1.0-0.2, 4.2-0.7)--(0.9-0.2, 4.2-0.7)--(0.9-0.2, 4.3-0.7)--(1.1-0.2, 4.3-0.7)
     --(1.1-0.2, 3.4+1.1-0.9)--(0.9-0.2, 3.4+1.1-0.9)-- (0.9-0.2, 3.5+1.1-0.9)--(1.0-0.2, 3.5+1.1-0.9)--(1.0-0.2, 3.6+1.1-0.9)--(0.8-0.2, 3.6+1.1-0.9)--(0.8-0.2, 3.8+1.1-0.9)--(0.7-0.2, 3.8+1.1-0.9)--(0.7-0.2, 3.7+1.1-0.9)--(0.6-0.2, 3.7+1.1-0.9)--(0.6-0.2, 4.0+1.1-0.9)--(0.7-0.2, 4.0+1.1-0.9)--(0.7-0.2, 3.9+1.1-0.9)--(0.8-0.2, 3.9+1.1-0.9)--(0.8-0.2, 4.1+1.1-0.9)--(1.0-0.2,4.1+1.1-0.9)--(1.0-0.2, 4.2+1.1-0.9)--(0.9-0.2, 4.2+1.1-0.9)--(0.9-0.2, 4.3+1.1-0.9)--(1.1-0.2, 4.3+1.1-0.9)
    --(1.1-0.2,6.0-1.1)--(1-0.2, 6.0-1.1)--(1-0.2, 5.9-1.1)--(0.9-0.2, 5.9-1.1)--(0.9-0.2, 6.0-1.1)--(0.8-0.2, 6.0-1.1)--(0.8-0.2, 6.1-1.1)--(1.1-0.2,6.1-1.1)
    --(0.9,5.4)--(0,5.4)--(0,0);

    \path[fill=gray!20] (3.6,3.6) circle (3.6);
\end{scope}

\begin{scope}
    \clip (1.8, 6.3)--(1.8, 7.2)--(7.2, 7.2)--(7.2, 1.8)--(6.3, 1.8)
    --(7.7-1.4, 2.7-0.5)--(8.0-1.4, 2.7-0.5)--(8.0-1.4, 2.8-0.5)--(7.9-1.4, 2.8-0.5)--(7.9-1.4, 2.9-0.5)--(7.8-1.4, 2.9-0.5)--(7.8-1.4, 2.8-0.5)--(7.7-1.4, 2.8-0.5)
    --(7.7-1.4, 3.4-0.7)--(7.9-1.4, 3.4-0.7)--(7.9-1.4, 3.5-0.7)--(7.8-1.4, 3.5-0.7)--(7.8-1.4, 3.6-0.7)--(8.0-1.4, 3.6-0.7)--(8.0-1.4, 3.8-0.7)--(8.1-1.4, 3.8-0.7)--(8.1-1.4, 3.7-0.7)--(8.2-1.4, 3.7-0.7)--(8.2-1.4, 4-0.7)--(8.1-1.4, 4-0.7)--(8.1-1.4, 3.9-0.7)--(8.0-1.4, 3.9-0.7)--(8.0-1.4, 4.1-0.7)--(7.8-1.4, 4.1-0.7)--(7.8-1.4, 4.2-0.7)--(7.9-1.4, 4.2-0.7)--(7.9-1.4, 4.3-0.7)--(7.7-1.4, 4.3-0.7)
    --(7.7-1.4, 3.4+0.2)--(7.9-1.4, 3.4+0.2)--(7.9-1.4, 3.5+0.2)--(7.8-1.4, 3.5+0.2)--(7.8-1.4, 3.6+0.2)--(8.0-1.4, 3.6+0.2)--(8.0-1.4, 3.8+0.2)--(8.1-1.4, 3.8+0.2)--(8.1-1.4, 3.7+0.2)--(8.2-1.4, 3.7+0.2)--(8.2-1.4, 4+0.2)--(8.1-1.4, 4+0.2)--(8.1-1.4, 3.9+0.2)--(8.0-1.4, 3.9+0.2)--(8.0-1.4, 4.1+0.2)--(7.8-1.4, 4.1+0.2)--(7.8-1.4, 4.2+0.2)--(7.9-1.4, 4.2+0.2)--(7.9-1.4, 4.3+0.2)--(7.7-1.4, 4.3+0.2)
    --(6.3, 4.5)--(4.5,4.5)--(4.5,6.3)
    --(5.4-0.9, 7.7-1.4)--(5.4-0.9, 7.9-1.4)--(5.3-0.9, 7.9-1.4)--(5.3-0.9, 7.8-1.4)--(5.2-0.9, 7.8-1.4)--(5.2-0.9, 8.0-1.4)--(5-0.9, 8.0-1.4)--(5-0.9, 8.1-1.4)--(5.1-0.9, 8.1-1.4)--(5.1-0.9, 8.2-1.4)--(4.8-0.9, 8.2-1.4)--(4.8-0.9, 8.1-1.4)--(4.9-0.9, 8.1-1.4)--(4.9-0.9, 8-1.4)--(4.7-0.9, 8-1.4)--(4.7-0.9, 7.8-1.4)--(4.6-0.9, 7.8-1.4)--(4.6-0.9, 7.9-1.4)--(4.5-0.9, 7.9-1.4)--(4.5-0.9, 7.7-1.4)
    --(5.4-1.1-0.7, 7.7-1.4)--(5.4-1.1-0.7, 7.9-1.4)--(5.3-1.1-0.7, 7.9-1.4)--(5.3-1.1-0.7, 7.8-1.4)--(5.2-1.1-0.7, 7.8-1.4)--(5.2-1.1-0.7, 8.0-1.4)--(5-1.1-0.7, 8.0-1.4)--(5-1.1-0.7, 8.1-1.4)--(5.1-1.1-0.7, 8.1-1.4)--(5.1-1.1-0.7, 8.2-1.4)--(4.8-1.1-0.7, 8.2-1.4)--(4.8-1.1-0.7, 8.1-1.4)--(4.9-1.1-0.7, 8.1-1.4)--(4.9-1.1-0.7, 8-1.4)--(4.7-1.1-0.7, 8-1.4)--(4.7-1.1-0.7, 7.8-1.4)--(4.6-1.1-0.7, 7.8-1.4)--(4.6-1.1-0.7, 7.9-1.4)--(4.5-1.1-0.7, 7.9-1.4)--(4.5-1.1-0.7, 7.7-1.4)
    --(2.8-0.5, 7.7-1.4)--(2.8-0.5,7.8-1.4)--(2.9-0.5,7.8-1.4)--(2.9-0.5,7.9-1.4)--(2.8-0.5,7.9-1.4)--(2.8-0.5,8.0-1.4)--(2.7-0.5,8.0-1.4)--(2.7-0.5,7.7-1.4)
    --(1.8, 6.3);
    
    \path[fill=gray!20] (3.6,3.6) circle (3.6);
\end{scope}

\begin{scope}
    \clip (3.6,3.6) circle (3.6);
\foreach \x in {1,...,7}  
{ 
\draw [ color=black!60] (0.9*\x,0)--(0.9*\x,7.2);
\draw [ color=black!60] (0,0.9*\x)--(7.2,0.9*\x);
}
\draw [very thick](5.4,0)--(5.4,0.9);\draw  [very thick](0,5.4)--(0.9,5.4); \draw  [very thick](2.7,0.9)--(2.7,2.7)--(0.9,2.7);
\draw  [very thick] (6.3, 4.5)--(4.5,4.5)--(4.5,6.3); \draw  [very thick] (7.2, 1.8)--(6.3, 1.8); \draw  [very thick] (1.8, 6.3)--(1.8, 7.2);
\end{scope}


\draw [color=orange, very thick] (5.4,0.9)--(6.1-1.1,1.1-0.2)--(6.1-1.1,0.8-0.2)--(6.0-1.1, 0.8-0.2)--(6.0-1.1,0.9-0.2)--(5.9-1.1,0.9-0.2)--(5.9-1.1,1.0-0.2)--(6.0-1.1,1.0-0.2)--(6.0-1.1,1.1-0.2)
    --(5.4-0.9,1.1-0.2)--(5.4-0.9,0.9-0.2)--(5.3-0.9,0.9-0.2)--(5.3-0.9,1.0-0.2)--(5.2-0.9,1.0-0.2)--(5.2-0.9, 0.8-0.2)--(5.0-0.9, 0.8-0.2)--(5.0-0.9, 0.7-0.2)--(5.1-0.9,0.7-0.2)--(5.1-0.9,0.6-0.2)--(4.8-0.9,0.6-0.2)--(4.8-0.9,0.7-0.2)--(4.9-0.9,0.7-0.2)--(4.9-0.9,0.8-0.2)--(4.7-0.9,0.8-0.2)--(4.7-0.9,1.0-0.2)--(4.6-0.9,1.0-0.2)--(4.6-0.9,0.9-0.2)--(4.5-0.9,0.9-0.2)
     --(5.4-1.8,0.9-0.2)--(5.3-1.8,0.9-0.2)--(5.3-1.8,1.0-0.2)--(5.2-1.8,1.0-0.2)--(5.2-1.8, 0.8-0.2)--(5.0-1.8, 0.8-0.2)--(5.0-1.8, 0.7-0.2)--(5.1-1.8,0.7-0.2)--(5.1-1.8,0.6-0.2)--(4.8-1.8,0.6-0.2)--(4.8-1.8,0.7-0.2)--(4.9-1.8,0.7-0.2)--(4.9-1.8,0.8-0.2)--(4.7-1.8,0.8-0.2)--(4.7-1.8,1.0-0.2)--(4.6-1.8, 1.0-0.2)--(4.6-1.8,0.9-0.2)--(4.5-1.8,0.9-0.2)--(4.5-1.8,1.1-0.2)
    --(2.7,0.9);

\draw [color=orange, very thick]    (0.9,2.7)
    --(1.1-0.2, 3.4-0.7)--(0.9-0.2, 3.4-0.7)-- (0.9-0.2, 3.5-0.7)--(1.0-0.2, 3.5-0.7)--(1.0-0.2, 3.6-0.7)--(0.8-0.2, 3.6-0.7)--(0.8-0.2, 3.8-0.7)--(0.7-0.2, 3.8-0.7)--(0.7-0.2, 3.7-0.7)--(0.6-0.2, 3.7-0.7)--(0.6-0.2, 4.0-0.7)--(0.7-0.2, 4.0-0.7)--(0.7-0.2, 3.9-0.7)--(0.8-0.2, 3.9-0.7)--(0.8-0.2, 4.1-0.7)--(1.0-0.2,4.1-0.7)--(1.0-0.2, 4.2-0.7)--(0.9-0.2, 4.2-0.7)--(0.9-0.2, 4.3-0.7)
     --(0.9-0.2, 3.4+1.1-0.9)-- (0.9-0.2, 3.5+1.1-0.9)--(1.0-0.2, 3.5+1.1-0.9)--(1.0-0.2, 3.6+1.1-0.9)--(0.8-0.2, 3.6+1.1-0.9)--(0.8-0.2, 3.8+1.1-0.9)--(0.7-0.2, 3.8+1.1-0.9)--(0.7-0.2, 3.7+1.1-0.9)--(0.6-0.2, 3.7+1.1-0.9)--(0.6-0.2, 4.0+1.1-0.9)--(0.7-0.2, 4.0+1.1-0.9)--(0.7-0.2, 3.9+1.1-0.9)--(0.8-0.2, 3.9+1.1-0.9)--(0.8-0.2, 4.1+1.1-0.9)--(1.0-0.2,4.1+1.1-0.9)--(1.0-0.2, 4.2+1.1-0.9)--(0.9-0.2, 4.2+1.1-0.9)--(0.9-0.2, 4.3+1.1-0.9)--(1.1-0.2, 4.3+1.1-0.9)
    --(1.1-0.2,6.0-1.1)--(1-0.2, 6.0-1.1)--(1-0.2, 5.9-1.1)--(0.9-0.2, 5.9-1.1)--(0.9-0.2, 6.0-1.1)--(0.8-0.2, 6.0-1.1)--(0.8-0.2, 6.1-1.1)--(1.1-0.2,6.1-1.1)
    --(0.9,5.4);

\draw [color=orange, very thick]    (6.3, 1.8)
    --(7.7-1.4, 2.7-0.5)--(8.0-1.4, 2.7-0.5)--(8.0-1.4, 2.8-0.5)--(7.9-1.4, 2.8-0.5)--(7.9-1.4, 2.9-0.5)--(7.8-1.4, 2.9-0.5)--(7.8-1.4, 2.8-0.5)--(7.7-1.4, 2.8-0.5)
    --(7.7-1.4, 3.4-0.7)--(7.9-1.4, 3.4-0.7)--(7.9-1.4, 3.5-0.7)--(7.8-1.4, 3.5-0.7)--(7.8-1.4, 3.6-0.7)--(8.0-1.4, 3.6-0.7)--(8.0-1.4, 3.8-0.7)--(8.1-1.4, 3.8-0.7)--(8.1-1.4, 3.7-0.7)--(8.2-1.4, 3.7-0.7)--(8.2-1.4, 4-0.7)--(8.1-1.4, 4-0.7)--(8.1-1.4, 3.9-0.7)--(8.0-1.4, 3.9-0.7)--(8.0-1.4, 4.1-0.7)--(7.8-1.4, 4.1-0.7)--(7.8-1.4, 4.2-0.7)--(7.9-1.4, 4.2-0.7)--(7.9-1.4, 4.3-0.7)
    --(7.9-1.4, 3.4+0.2)--(7.9-1.4, 3.5+0.2)--(7.8-1.4, 3.5+0.2)--(7.8-1.4, 3.6+0.2)--(8.0-1.4, 3.6+0.2)--(8.0-1.4, 3.8+0.2)--(8.1-1.4, 3.8+0.2)--(8.1-1.4, 3.7+0.2)--(8.2-1.4, 3.7+0.2)--(8.2-1.4, 4+0.2)--(8.1-1.4, 4+0.2)--(8.1-1.4, 3.9+0.2)--(8.0-1.4, 3.9+0.2)--(8.0-1.4, 4.1+0.2)--(7.8-1.4, 4.1+0.2)--(7.8-1.4, 4.2+0.2)--(7.9-1.4, 4.2+0.2)--(7.9-1.4, 4.3+0.2)--(7.7-1.4, 4.3+0.2)
    --(6.3, 4.5);

\draw [color=orange, very thick]    (4.5,6.3)
    --(5.4-0.9, 7.7-1.4)--(5.4-0.9, 7.9-1.4)--(5.3-0.9, 7.9-1.4)--(5.3-0.9, 7.8-1.4)--(5.2-0.9, 7.8-1.4)--(5.2-0.9, 8.0-1.4)--(5-0.9, 8.0-1.4)--(5-0.9, 8.1-1.4)--(5.1-0.9, 8.1-1.4)--(5.1-0.9, 8.2-1.4)--(4.8-0.9, 8.2-1.4)--(4.8-0.9, 8.1-1.4)--(4.9-0.9, 8.1-1.4)--(4.9-0.9, 8-1.4)--(4.7-0.9, 8-1.4)--(4.7-0.9, 7.8-1.4)--(4.6-0.9, 7.8-1.4)--(4.6-0.9, 7.9-1.4)--(4.5-0.9, 7.9-1.4)
    --(5.4-1.1-0.7, 7.9-1.4)--(5.3-1.1-0.7, 7.9-1.4)--(5.3-1.1-0.7, 7.8-1.4)--(5.2-1.1-0.7, 7.8-1.4)--(5.2-1.1-0.7, 8.0-1.4)--(5-1.1-0.7, 8.0-1.4)--(5-1.1-0.7, 8.1-1.4)--(5.1-1.1-0.7, 8.1-1.4)--(5.1-1.1-0.7, 8.2-1.4)--(4.8-1.1-0.7, 8.2-1.4)--(4.8-1.1-0.7, 8.1-1.4)--(4.9-1.1-0.7, 8.1-1.4)--(4.9-1.1-0.7, 8-1.4)--(4.7-1.1-0.7, 8-1.4)--(4.7-1.1-0.7, 7.8-1.4)--(4.6-1.1-0.7, 7.8-1.4)--(4.6-1.1-0.7, 7.9-1.4)--(4.5-1.1-0.7, 7.9-1.4)--(4.5-1.1-0.7, 7.7-1.4)
    --(2.8-0.5, 7.7-1.4)--(2.8-0.5,7.8-1.4)--(2.9-0.5,7.8-1.4)--(2.9-0.5,7.9-1.4)--(2.8-0.5,7.9-1.4)--(2.8-0.5,8.0-1.4)--(2.7-0.5,8.0-1.4)--(2.7-0.5,7.7-1.4)
    --(1.8, 6.3);

\draw[color=violet!60, very thick] (3.6,3.6) circle (3.6);

\end{tikzpicture}
\end{center}
\caption{Part of the jaw (zoom in).}\label{fig_jaw_zoom_in}
\end{figure}


\begin{figure}[H]
\begin{center}
\begin{tikzpicture}[scale=0.1]

\draw [fill=gray!20] (0,63)--(18,63)
--(18,54)--(27-5,66-12)--(27-5,69-12)--(28-5,69-12)--(28-5,68-12)--(29-5,68-12)--(29-5,67-12)--(28-5,67-12)--(28-5,66-12)--(27,54)
--(34-7,66-12)--(34-7,68-12)--(35-7,68-12)--(35-7,67-12)--(36-7,67-12)--(36-7,69-12)--(38-7,69-12)--(38-7,70-12)--(37-7,70-12)--(37-7,71-12)--(40-7,71-12)--(40-7,70-12)--(39-7,70-12)--(39-7,69-12)--(41-7,69-12)--(41-7,67-12)--(42-7,67-12)--(42-7,68-12)--(43-7,68-12)
--(45-9,68-12)--(46-9,68-12)--(46-9,67-12)--(47-9,67-12)--(47-9,69-12)--(49-9,69-12)--(49-9,70-12)--(48-9,70-12)--(48-9,71-12)--(51-9,71-12)--(51-9,70-12)--(50-9,70-12)--(50-9,69-12)--(52-9,69-12)--(52-9,67-12)--(53-9,67-12)--(53-9,68-12)--(54-9,68-12)--(54-9,66-12)
--(45,54)--(45,36)--(63,36)
--(77-14,43-7)--(79-14,43-7)--(79-14,42-7)--(78-14,42-7)--(78-14,41-7)--(80-14,41-7)--(80-14,39-7)--(81-14,39-7)--(81-14,40-7)--(82-14,40-7)--(82-14,37-7)--(81-14,37-7)--(81-14,38-7)--(80-14,38-7)--(80-14,36-7)--(78-14,36-7)--(78-14,35-7)--(79-14,35-7)--(79-14,34-7)
--(79-14,32-5)--(79-14,31-5)--(78-14,31-5)--(78-14,30-5)--(80-14,30-5)--(80-14,28-5)--(81-14,28-5)--(81-14,29-5)--(82-14,29-5)--(82-14,26-5)--(81-14,26-5)--(81-14,27-5)--(80-14,27-5)--(80-14,25-5)--(78-14,25-5)--(78-14,24-5)--(79-14,24-5)--(79-14,23-5)--(77-14,23-5)
--(63,18)
--(77-14,17-3)--(78-14,17-3)--(78-14,18-3)--(79-14,18-3)--(79-14,17-3)--(80-14,17-3)--(80-14,16-3)--(77-14,16-3)
--(63,9)--(54,9)--(54,0)
--(61-11,0)--(61-11,-3)--(60-11,-3)--(60-11,-2)--(59-11,-2)--(59-11,-1)--(60-11,-1)--(60-11,0)
--(45,0)
--(54-9,0)--(54-9,-2)--(53-9,-2)--(53-9,-1)--(52-9,-1)--(52-9,-3)--(50-9,-3)--(50-9,-4)--(51-9,-4)--(51-9,-5)--(48-9,-5)--(48-9,-4)--(49-9,-4)--(49-9,-3)--(47-9,-3)--(47-9,-1)--(46-9,-1)--(46-9,-2)--(45-9,-2)
--(43-7,-2)--(42-7,-2)--(42-7,-1)--(41-7,-1)--(41-7,-3)--(39-7,-3)--(39-7,-4)--(40-7,-4)--(40-7,-5)--(37-7,-5)--(37-7,-4)--(38-7,-4)--(38-7,-3)--(36-7,-3)--(36-7,-1)--(35-7,-1)--(35-7,-2)--(34-7,-2)--(34-7,0)
--(27,0)
--(27,18)--(9,18)
--(11-2,23-5)--(9-2,23-5)--(9-2,24-5)--(10-2,24-5)--(10-2,25-5)--(8-2,25-5)--(8-2,27-5)--(7-2,27-5)--(7-2,26-5)--(6-2,26-5)--(6-2,29-5)--(7-2,29-5)--(7-2,28-5)--(8-2,28-5)--(8-2,30-5)--(10-2,30-5)--(10-2,31-5)--(9-2,31-5)--(9-2,32-5)
--(9-2,34-7)--(9-2,35-7)--(10-2,35-7)--(10-2,36-7)--(8-2,36-7)--(8-2,38-7)--(7-2,38-7)--(7-2,37-7)--(6-2,37-7)--(6-2,40-7)--(7-2,40-7)--(7-2,39-7)--(8-2,39-7)--(8-2,41-7)--(10-2,41-7)--(10-2,42-7)--(9-2,42-7)--(9-2,43-7)--(11-2,43-7)
--(9,36)
--(11-2,49-9)--(10-2,49-9)--(10-2,48-9)--(9-2,48-9)--(9-2,49-9)--(8-2,49-9)--(8-2,50-9)--(11-2,50-9)
--(9,45)--(0,45)--(0,63);

\draw [color=black!60] (0,54)--(18,54)--(18,18);
\draw [color=black!60] (9,63)--(9,45)--(45,45);
\draw [color=black!60] (9,36)--(45,36)--(45,0);
\draw [color=black!60] (7,27)--(65,27);
\draw [color=black!60] (27,54)--(27,18)--(63,18);
\draw [color=black!60] (36,56)--(36,-2);
\draw [color=black!60] (27,9)--(54,9)--(54,36);
\draw [color=black!60] (9,18)--(9,36);
\draw [color=black!60] (63,18)--(63,36);
\draw [color=black!60] (27,0)--(45,0);
\draw [color=black!60] (27,54)--(45,54);

\end{tikzpicture}
\end{center}
\caption{A filler.}\label{fig_filler}
\end{figure}
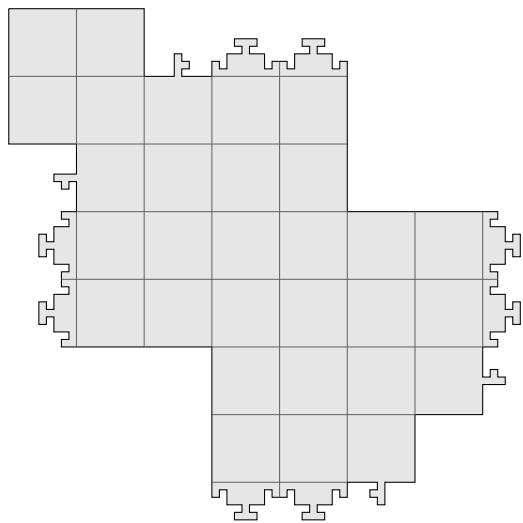

After the meats have been placed inside the jaws, there are still several holes inside the jaws. These holes are filled by the \textit{fillers} (Figure \ref{fig_filler}) and the \textit{teeth} (Figure \ref{fig_teeth}). The bigger holes are fill by the fillers. Therefore, the shape of the base polyomino and the shape of the bumps of the filler match exactly the shape of a portion inside the jaw and the shape of the dents there. See the tiling pattern in Figure \ref{fig_pattern} for how the fillers fill the holes. In general, the base boundary word for the filler is (starting from the upper left corner)
$$d^{18} r^9 (d^9)^{t+1} r^{18}r^{18} (r^9)^{t+1} u^9 r^9 (u^9)^{t+1} l^{18} u^{18} (l^9)^{t+1} u^9 l^{18}.$$

Besides the bigger holes, there are still small gaps between the deeper dents of the jaw and the normal or deeper dents of the meat. We need $3$ kinds of teeth (see Figure \ref{fig_teeth}) to fill the gaps. Note that in Figure \ref{fig_teeth} and Figure \ref{fig_links}, each small gray square represents a $1\times 1$ square, rather than a $9\times 9$ square in the figures of the meat, the jaw and the filler. The first tooth can fill the gap between a normal dent and a deeper dent, in either a vertical or horizontal direction. The second tooth can fill the gap between two vertical deeper dents, and the third tooth can fill the gap between two horizontal deeper dents. So the purpose of the deeper dents we add to the jaw is to decrease the number of teeth needed to fill the gaps. We would have needed $4$ different teeth if those deeper dents were not presented in the jaw, or the dents were in other shapes. The teeth are the smallest ones in our set of $8$ polyominoes. Their sizes and shapes are fixed and does not change with $n$ or $m$ (the number of Wang tiles or colors).


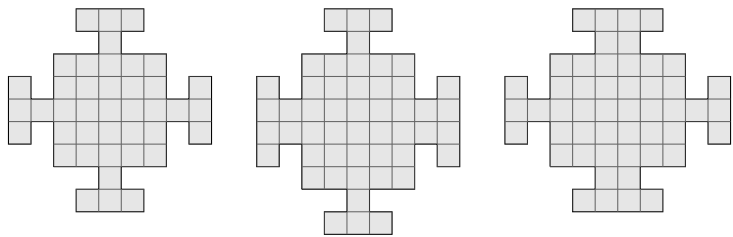
\begin{figure}[H]
\begin{center}
\begin{tikzpicture}[scale=0.3]

\draw [fill=gray!20] (0,0)--(2,0)
--(2,-1)--(1,-1)--(1,-2)--(4,-2)--(4,-1)--(3,-1)
--(3,0)--(5,0)
--(5,2)--(6,2)--(6,1)--(7,1)--(7,4)--(6,4)--(6,3)--(5,3)
--(5,5)
--(3,5)--(3,6)--(4,6)--(4,7)--(1,7)--(1,6)--(2,6)--(2,5)
--(0,5)
--(0,3)--(-1,3)--(-1,4)--(-2,4)--(-2,1)--(-1,1)--(-1,2)--(0,2)
--(0,0);

\draw [fill=gray!20] (11+0,-1)--(11+2,-1)
--(11+2,-2)--(11+1,-2)--(11+1,-3)--(11+4,-3)--(11+4,-2)--(11+3,-2)
--(11+3,-1)--(11+5,-1)
--(11+5,1)--(11+6,1)--(11+6,0)--(11+7,0)--(11+7,4)--(11+6,4)--(11+6,3)--(11+5,3)
--(11+5,5)
--(11+3,5)--(11+3,6)--(11+4,6)--(11+4,7)--(11+1,7)--(11+1,6)--(11+2,6)--(11+2,5)
--(11+0,5)
--(11+0,3)--(11+-1,3)--(11+-1,4)--(11+-2,4)--(11+-2,0)--(11+-1,0)--(11+-1,1)--(11+0,1)
--(11+0,-1);

\draw [fill=gray!20] (22+0,0)--(22+2,0)
--(22+2,-1)--(22+1,-1)--(22+1,-2)--(23+4,-2)--(23+4,-1)--(23+3,-1)
--(23+3,0)--(23+5,0)
--(23+5,2)--(23+6,2)--(23+6,1)--(23+7,1)--(23+7,4)--(23+6,4)--(23+6,3)--(23+5,3)
--(23+5,5)
--(23+3,5)--(23+3,6)--(23+4,6)--(23+4,7)--(22+1,7)--(22+1,6)--(22+2,6)--(22+2,5)
--(22+0,5)
--(22+0,3)--(22+-1,3)--(22+-1,4)--(22+-2,4)--(22+-2,1)--(22+-1,1)--(22+-1,2)--(22+0,2)
--(22+0,0);

\begin{scope}
    \clip (0,0)--(2,0)
--(2,-1)--(1,-1)--(1,-2)--(4,-2)--(4,-1)--(3,-1)
--(3,0)--(5,0)
--(5,2)--(6,2)--(6,1)--(7,1)--(7,4)--(6,4)--(6,3)--(5,3)
--(5,5)
--(3,5)--(3,6)--(4,6)--(4,7)--(1,7)--(1,6)--(2,6)--(2,5)
--(0,5)
--(0,3)--(-1,3)--(-1,4)--(-2,4)--(-2,1)--(-1,1)--(-1,2)--(0,2)
--(0,0);

\foreach \x in {-1,...,6}  
{ 
\draw [ color=black!60] (\x,-2)--(\x,7);
\draw [ color=black!60] (-2,\x)--(7,\x);
}
\end{scope}

\begin{scope}
    \clip (11+0,-1)--(11+2,-1)
--(11+2,-2)--(11+1,-2)--(11+1,-3)--(11+4,-3)--(11+4,-2)--(11+3,-2)
--(11+3,-1)--(11+5,-1)
--(11+5,1)--(11+6,1)--(11+6,0)--(11+7,0)--(11+7,4)--(11+6,4)--(11+6,3)--(11+5,3)
--(11+5,5)
--(11+3,5)--(11+3,6)--(11+4,6)--(11+4,7)--(11+1,7)--(11+1,6)--(11+2,6)--(11+2,5)
--(11+0,5)
--(11+0,3)--(11+-1,3)--(11+-1,4)--(11+-2,4)--(11+-2,0)--(11+-1,0)--(11+-1,1)--(11+0,1)
--(11+0,-1);

\foreach \x in {-2,...,6}  
{ 
\draw [ color=black!60] (11+\x,-3)--(11+\x,7);
\draw [ color=black!60] (9,\x)--(18,\x);
}
\end{scope}

\begin{scope}
    \clip (22+0,0)--(22+2,0)
--(22+2,-1)--(22+1,-1)--(22+1,-2)--(23+4,-2)--(23+4,-1)--(23+3,-1)
--(23+3,0)--(23+5,0)
--(23+5,2)--(23+6,2)--(23+6,1)--(23+7,1)--(23+7,4)--(23+6,4)--(23+6,3)--(23+5,3)
--(23+5,5)
--(23+3,5)--(23+3,6)--(23+4,6)--(23+4,7)--(22+1,7)--(22+1,6)--(22+2,6)--(22+2,5)
--(22+0,5)
--(22+0,3)--(22+-1,3)--(22+-1,4)--(22+-2,4)--(22+-2,1)--(22+-1,1)--(22+-1,2)--(22+0,2)
--(22+0,0);

\foreach \x in {-1,...,7}  
{ 
\draw [ color=black!60] (22+\x,-2)--(22+\x,8);
\draw [ color=black!60] (20,\x)--(30,\x);
}
\end{scope}

\end{tikzpicture}
\end{center}
\caption{The teeth.}\label{fig_teeth}
\end{figure}

The shape of normal and deeper dents to encode the colors of Wang tiles also decreases the number of link polyominoes needed to connect the meats outside the jaws. Now we just need $2$ \textit{links}, a horizontal link and a vertical link (see Figure \ref{fig_links}), compared to $4$ links in Ollinger's proof for the undecidability of $11$-polyomino tiling problem. The base polyomino the the horizontal link is a rectangle of size $\Bigl (18(n-1)(t+4)+9 \Bigr )\times 9$, as illustrated by the red rectangle in Figure \ref{fig_links}. In other words, the boundary word of the base polyomino of the horizontal link is $d^9(r^9)^{2(n-1)(t+4)+1}u^9(l^9)^{2(n-1)(t+4)+1}$. Two bumps that match the shapes of the normal dent and the deeper dent are added to the left side and the right side of the base polyomino, respectively, to form the horizontal link. We can rotate the horizontal link by $\pi/2$ to get the vertical link.


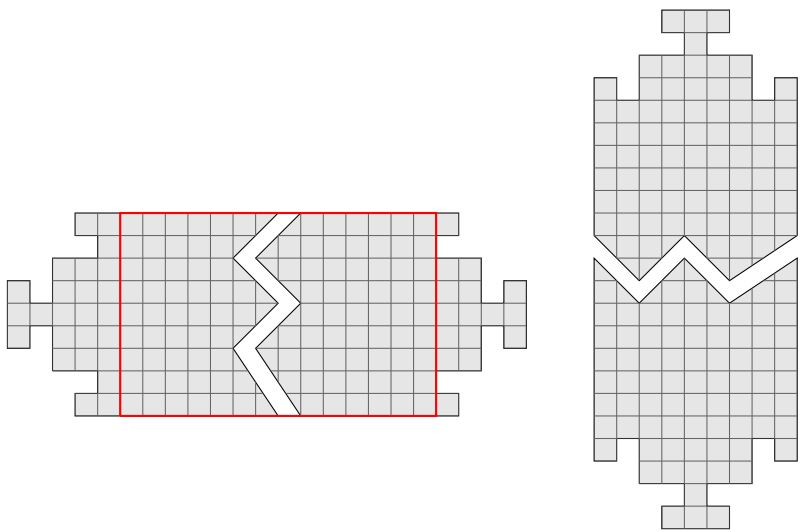
\begin{figure}[H]
\begin{center}
\begin{tikzpicture}[scale=0.3]

\draw [fill=gray!20] (0,0)--(2,0)--(2,-1)--(1,-1)--(1,-2)--(10,-2)
--(8,1)--(10,3)--(8,5)
--(10,7)--(1,7)--(1,6)--(2,6)--(2,5)--(0,5)
--(0,3)--(-1,3)--(-1,4)--(-2,4)--(-2,1)--(-1,1)--(-1,2)--(0,2)
--(0,0);
\draw [fill=gray!20] (11,-2)--(18,-2)--(18,-1)--(17,-1)--(17,0)--(19,0)
--(19,2)--(20,2)--(20,1)--(21,1)--(21,4)--(20,4)--(20,3)--(19,3)
--(19,5)--(17,5)--(17,6)--(18,6)--(18,7)--(11,7)
--(9,5)--(11,3)--(9,1)
--(11,-2);

\draw [fill=gray!20] (26,0-5)--(28,0-5)--(28,-1-5)--(27,-1-5)--(27,-2-5)--(30,-2-5)--(30,-1-5)--(29,-1-5)--(29,0-5)--(31,0-5)
--(31,2-5)--(32,2-5)--(32,1-5)--(33,1-5)--(33,10-5)
--(30,3)--(28,5)--(26,3)
--(24,10-5)--(24,1-5)--(25,1-5)--(25,2-5)--(26,2-5)--(26,0-5);

\draw [fill=gray!20] (33,6)--(33,13)--(32,13)--(32,12)--(31,12)
--(31,14)
--(29,14)--(29,15)--(30,15)--(30,16)--(27,16)--(27,15)--(28,15)--(28,14)
--(26,14)--(26,12)--(25,12)--(25,13)--(24,13)--(24,6)
--(26,4)--(28,6)--(30,4)--(33,6);

\begin{scope}
    \clip (0,0)--(2,0)--(2,-1)--(1,-1)--(1,-2)--(10,-2)
--(8,1)--(10,3)--(8,5)
--(10,7)--(1,7)--(1,6)--(2,6)--(2,5)--(0,5)
--(0,3)--(-1,3)--(-1,4)--(-2,4)--(-2,1)--(-1,1)--(-1,2)--(0,2)
--(0,0);

\foreach \x in {-1,...,6}  
{ 
\draw [ color=black!60] (-2,\x)--(21,\x);
}
\foreach \x in {-2,...,10}  
{ 
\draw [ color=black!60] (\x, -2)--(\x, 7);
}
\end{scope}

\begin{scope}
    \clip (11,-2)--(18,-2)--(18,-1)--(17,-1)--(17,0)--(19,0)
--(19,2)--(20,2)--(20,1)--(21,1)--(21,4)--(20,4)--(20,3)--(19,3)
--(19,5)--(17,5)--(17,6)--(18,6)--(18,7)--(11,7)
--(9,5)--(11,3)--(9,1)
--(11,-2);

\foreach \x in {-1,...,6}  
{ 
\draw [ color=black!60] (-2,\x)--(21,\x);
}
\foreach \x in {10,...,21}  
{ 
\draw [ color=black!60] (\x, -2)--(\x, 7);
}
\end{scope}

\begin{scope}
    \clip (26,0-5)--(28,0-5)--(28,-1-5)--(27,-1-5)--(27,-2-5)--(30,-2-5)--(30,-1-5)--(29,-1-5)--(29,0-5)--(31,0-5)
--(31,2-5)--(32,2-5)--(32,1-5)--(33,1-5)--(33,10-5)
--(30,3)--(28,5)--(26,3)
--(24,10-5)--(24,1-5)--(25,1-5)--(25,2-5)--(26,2-5)--(26,0-5);

\foreach \x in {-7,...,5}  
{ 
\draw [ color=black!60] (24,\x)--(33,\x);
}
\foreach \x in {25,...,32}  
{ 
\draw [ color=black!60] (\x, -7)--(\x, 5);
}
\end{scope}

\begin{scope}
    \clip (33,6)--(33,13)--(32,13)--(32,12)--(31,12)
--(31,14)
--(29,14)--(29,15)--(30,15)--(30,16)--(27,16)--(27,15)--(28,15)--(28,14)
--(26,14)--(26,12)--(25,12)--(25,13)--(24,13)--(24,6)
--(26,4)--(28,6)--(30,4)--(33,6);

\foreach \x in {5,...,15}  
{ 
\draw [ color=black!60] (24,\x)--(33,\x);
}
\foreach \x in {25,...,32}  
{ 
\draw [ color=black!60] (\x, 4)--(\x, 16);
}
\end{scope}

\draw [color=red,thick] (3,-2)--(17,-2)--(17,7)--(3,7)--(3,-2);

\end{tikzpicture}
\end{center}
\caption{The links.}\label{fig_links}
\end{figure}

We have finished introducing all the $8$ polyominoes in our reduction. To complete the proof, we show that if this set of $8$ polyominoes can tile the plane, its tiling must adopt the pattern illustrated in Figure \ref{fig_pattern}. In Figure \ref{fig_pattern}, the gray tiles are jaws, the orange tiles are meats, the green tiles are fillers, and the light violet tiles are links. All the small bumps and dents are omitted in Figure \ref{fig_pattern}, so the teeth cannot be seen at all. Three of the meats are depicted in full, together with the fillers which fill the holes left by these three meats inside the jaws. For all the other meats, only the parts outside the jaws are depicted. So most parts inside the jaws (the fillers and parts of the meats) are omitted in Figure \ref{fig_pattern}, too. Because in this tiling pattern, the inside parts of the jaw can always be filled up. The simulated Wang tiles which are placed outside the jaws for every meat determine whether this pattern can be extended infinitely to tile the entire plane.


\begin{figure}[H]
\begin{center}
\begin{tikzpicture}[scale=0.1]

\foreach \x in {0,...,3}
\foreach \y in {0,...,3} 
{ 
\draw [ fill=gray!20] (31*\x+0,0+31*\y)--(31*\x+0,29+31*\y)--(31*\x+29,29+31*\y)--(31*\x+29,0+31*\y)--(31*\x+0,0+31*\y);
}

\foreach \x in {-1,...,3}
\foreach \y in {-1,...,3} 
{ 
\draw [fill=orange] (31*\x+29,29+31*\y)--(31*\x+29,27+31*\y)--(31*\x+31,27+31*\y)--(31*\x+31,29+31*\y)--(31*\x+33,29+31*\y)--(31*\x+33,31+31*\y)--(31*\x+31,31+31*\y)--(31*\x+31,33+31*\y)--(31*\x+29,33+31*\y)--(31*\x+29,31+31*\y)--(31*\x+27,31+31*\y)--(31*\x+27,29+31*\y)--(31*\x+29,29+31*\y);
}

\foreach \x in {0,31,62,93}
\foreach \y in {0,1,31,32,62,63} 
{ 
\draw [fill=violet!20] (\x+2,29+\y)--(\x+27,29+\y)--(\x+27,30+\y)--(\x+2,30+\y)--(\x+2,29+\y);
}

\foreach \y in {0,31,62,93}
\foreach \x in {0,1,31,32,62,63} 
{ 
\draw [fill=violet!20] (\x+29,2+\y)--(\x+29,27+\y)--(\x+30,27+\y)--(\x+30,2+\y)--(\x+29,2+\y);
}


\draw [fill=orange] (60,60)--(60,58)--(63,58)--(63,57)--(64,57)--(64,54)--(66,54)--(66,52)--(69,52)--(69,53)--(70,53)--(70,56)--(68,56)--(68,58)--(65,58)--(65,59)--(64,59)--(64,62)
--(62,62)--(62,64)--(59,64)--(59,65)--(58,65)--(58,68)--(56,68)--(56,70)--(53,70)--(53,71)--(51,71)--(51,69)--(52,69)--(52,66)--(54,66)--(54,64)--(57,64)--(57,63)--(58,63)--(58,60)--(60,60);

\foreach \x in {37}
\foreach \y in {-37} 
{ 
\draw [fill=orange] (\x+60,60+\y)--(\x+60,58+\y)--(\x+63,58+\y)--(\x+63,57+\y)--(\x+64,57+\y)--(\x+64,54+\y)--(\x+66,54+\y)--(\x+66,52+\y)--(\x+69,52+\y)--(\x+69,53+\y)--(\x+70,53+\y)--(\x+70,56+\y)--(\x+68,56+\y)--(\x+68,58+\y)--(\x+65,58+\y)--(\x+65,59+\y)--(\x+64,59+\y)--(\x+64,62+\y)
--(\x+62,62+\y)--(\x+62,64+\y)--(\x+59,64+\y)--(\x+59,65+\y)--(\x+58,65+\y)--(\x+58,68+\y)--(\x+56,68+\y)--(\x+56,70+\y)--(\x+53,70+\y)--(\x+53,71+\y)--(\x+51,71+\y)--(\x+51,69+\y)--(\x+52,69+\y)--(\x+52,66+\y)--(\x+54,66+\y)--(\x+54,64+\y)--(\x+57,64+\y)--(\x+57,63+\y)--(\x+58,63+\y)--(\x+58,60+\y)--(\x+60,60+\y);
}

\foreach \x in {25}
\foreach \y in {6} 
{ 
\draw [fill=orange] (\x+60,60+\y)--(\x+60,58+\y)--(\x+63,58+\y)--(\x+63,57+\y)--(\x+64,57+\y)--(\x+64,54+\y)--(\x+66,54+\y)--(\x+66,52+\y)--(\x+69,52+\y)--(\x+69,53+\y)--(\x+70,53+\y)--(\x+70,56+\y)--(\x+68,56+\y)--(\x+68,58+\y)--(\x+65,58+\y)--(\x+65,59+\y)--(\x+64,59+\y)--(\x+64,62+\y)
--(\x+62,62+\y)--(\x+62,64+\y)--(\x+59,64+\y)--(\x+59,65+\y)--(\x+58,65+\y)--(\x+58,68+\y)--(\x+56,68+\y)--(\x+56,70+\y)--(\x+53,70+\y)--(\x+53,71+\y)--(\x+51,71+\y)--(\x+51,69+\y)--(\x+52,69+\y)--(\x+52,66+\y)--(\x+54,66+\y)--(\x+54,64+\y)--(\x+57,64+\y)--(\x+57,63+\y)--(\x+58,63+\y)--(\x+58,60+\y)--(\x+60,60+\y);
}


\draw [fill=lime] (69,53)--(69,51)--(70,51)--(70,48)--(72,48)--(72,46)--(75,46)--(75,47)--(76,47)--(76,50)--(74,50)--(74,52)--(71,52)--(71,53)--(69,53);

\foreach \x in {-24}
\foreach \y in {24} 
{ 
\draw [fill=lime] (\x+69,53+\y)--(\x+69,51+\y)--(\x+70,51+\y)--(\x+70,48+\y)--(\x+72,48+\y)--(\x+72,46+\y)--(\x+75,46+\y)--(\x+75,47+\y)--(\x+76,47+\y)--(\x+76,50+\y)--(\x+74,50+\y)--(\x+74,52+\y)--(\x+71,52+\y)--(\x+71,53+\y)--(\x+69,53+\y);
}

\foreach \x in {7}
\foreach \y in {-7} 
{ 
\draw [fill=lime] (\x+69,53+\y)--(\x+69,51+\y)--(\x+70,51+\y)--(\x+70,48+\y)--(\x+72,48+\y)--(\x+72,46+\y)--(\x+75,46+\y)--(\x+75,47+\y)--(\x+76,47+\y)--(\x+76,50+\y)--(\x+74,50+\y)--(\x+74,52+\y)--(\x+71,52+\y)--(\x+71,53+\y)--(\x+69,53+\y);
}
\foreach \x in {13}
\foreach \y in {-13} 
{ 
\draw [fill=lime] (\x+69,53+\y)--(\x+69,51+\y)--(\x+70,51+\y)--(\x+70,48+\y)--(\x+72,48+\y)--(\x+72,46+\y)--(\x+75,46+\y)--(\x+75,47+\y)--(\x+76,47+\y)--(\x+76,50+\y)--(\x+74,50+\y)--(\x+74,52+\y)--(\x+71,52+\y)--(\x+71,53+\y)--(\x+69,53+\y);
}

\foreach \x in {25}
\foreach \y in {6} 
{ 
\draw [fill=lime] (\x+69,53+\y)--(\x+69,51+\y)--(\x+70,51+\y)--(\x+70,48+\y)--(\x+72,48+\y)--(\x+72,46+\y)--(\x+75,46+\y)--(\x+75,47+\y)--(\x+76,47+\y)--(\x+76,50+\y)--(\x+74,50+\y)--(\x+74,52+\y)--(\x+71,52+\y)--(\x+71,53+\y)--(\x+69,53+\y);
}
\foreach \x in {31}
\foreach \y in {0} 
{ 
\draw [fill=lime] (\x+69,53+\y)--(\x+69,51+\y)--(\x+70,51+\y)--(\x+70,48+\y)--(\x+72,48+\y)--(\x+72,46+\y)--(\x+75,46+\y)--(\x+75,47+\y)--(\x+76,47+\y)--(\x+76,50+\y)--(\x+74,50+\y)--(\x+74,52+\y)--(\x+71,52+\y)--(\x+71,53+\y)--(\x+69,53+\y);
}

\end{tikzpicture}
\end{center}
\caption{The tiling pattern.}\label{fig_pattern}
\end{figure}
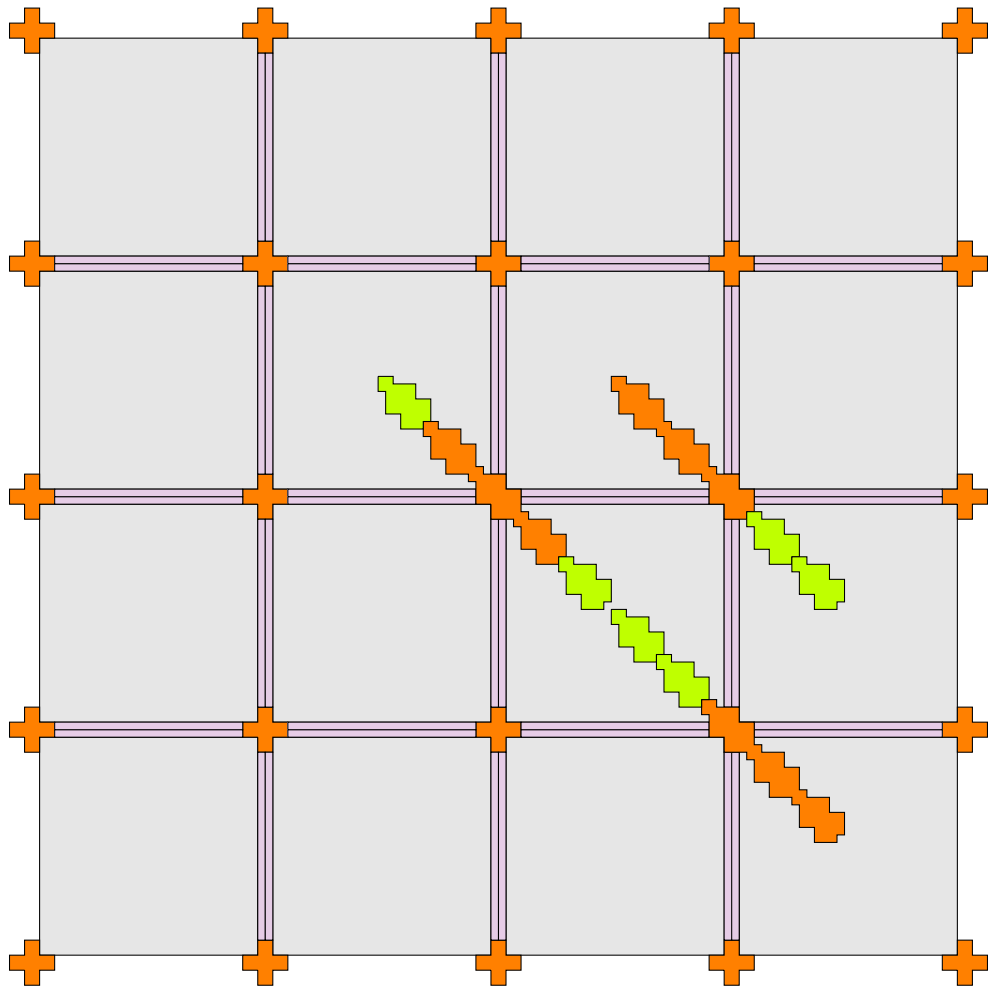

We prove that the pattern of Figure \ref{fig_pattern} must be adopted in order to tile the plane by the following steps of arguments.

\begin{itemize}
    \item Note that the fillers, the teeth, and the links have the same shape of bumps. Because of the existence of these bumps, we cannot tile the plane with only these three types of polyominoes. For example, to tile the location marked by a black dot next to a tooth (left of Figure \ref{fig_local}), we must put another tooth (or a link or a filler with this same shape of bumps). But after the placement of the second tile, the location marked by a red dot on the right of Figure \ref{fig_local} cannot be tile anymore.

    \item As a consequence, to tile the plane, we must use the jaw or the meat. If we use the jaw, then we must also use the meat, because the jaw has $T$-shape dents which can only be matched by the $T$-shape bumps of the meat. So in all, we must use meat in tiling.


\begin{figure}[H]
\begin{center}
\begin{tikzpicture}[scale=0.3]

\foreach \x in {0,16}
\foreach \y in {0} 
{ 
\draw [fill=gray!20] (\x+0,0+\y)--(\x+2,0+\y)
--(\x+2,-1+\y)--(\x+1,-1+\y)--(\x+1,-2+\y)--(\x+4,-2+\y)--(\x+4,-1+\y)--(\x+3,-1+\y)
--(\x+3,0+\y)--(\x+5,0+\y)
--(\x+5,2+\y)--(\x+6,2+\y)--(\x+6,1+\y)--(\x+7,1+\y)--(\x+7,4+\y)--(\x+6,4+\y)--(\x+6,3+\y)--(\x+5,3+\y)
--(\x+5,5+\y)
--(\x+3,5+\y)--(\x+3,6+\y)--(\x+4,6+\y)--(\x+4,7+\y)--(\x+1,7+\y)--(\x+1,6+\y)--(\x+2,6+\y)--(\x+2,5+\y)
--(\x+0,5+\y)
--(\x+0,3+\y)--(\x+-1,3+\y)--(\x+-1,4+\y)--(\x+-2,4+\y)--(\x+-2,1+\y)--(\x+-1,1+\y)--(\x+-1,2+\y)--(\x+0,2+\y)
--(\x+0,0+\y);
}

\foreach \x in {23}
\foreach \y in {-2} 
{ 
\draw [fill=gray!20] (\x+0,0+\y)--(\x+2,0+\y)
--(\x+2,-1+\y)--(\x+1,-1+\y)--(\x+1,-2+\y)--(\x+4,-2+\y)--(\x+4,-1+\y)--(\x+3,-1+\y)
--(\x+3,0+\y)--(\x+5,0+\y)
--(\x+5,2+\y)--(\x+6,2+\y)--(\x+6,1+\y)--(\x+7,1+\y)--(\x+7,4+\y)--(\x+6,4+\y)--(\x+6,3+\y)--(\x+5,3+\y)
--(\x+5,5+\y)
--(\x+3,5+\y)--(\x+3,6+\y)--(\x+4,6+\y)--(\x+4,7+\y)--(\x+1,7+\y)--(\x+1,6+\y)--(\x+2,6+\y)--(\x+2,5+\y)
--(\x+0,5+\y)
--(\x+0,3+\y)--(\x+-1,3+\y)--(\x+-1,4+\y)--(\x+-2,4+\y)--(\x+-2,1+\y)--(\x+-1,1+\y)--(\x+-1,2+\y)--(\x+0,2+\y)
--(\x+0,0+\y);
}

\filldraw[black] (5.5,1.5) circle (.2);
\filldraw[red] (19.5,-0.5) circle (.2);

\draw  [->, line width=3,  red] (8,2.5)--(13,2.5);

\end{tikzpicture}
\end{center}
\caption{The local placement of teeth.}\label{fig_local}
\end{figure}
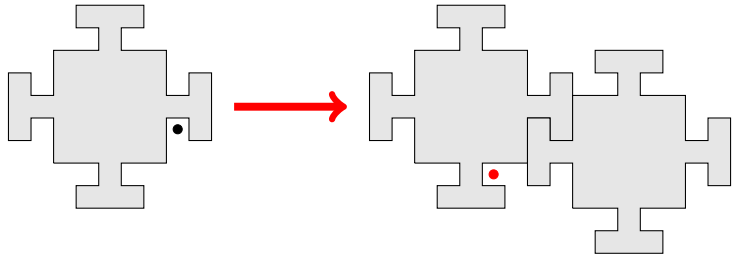

    \item As we have mentioned before, a meat can only be embraced from two opposite directions by two jaws, leaving a simulated Wang tile outside the jaws. Now, consider the dents of the simulated Wang tile outside the jaws. By a similar local argument as in Figure \ref{fig_local}, if those dents are to be matched by a tooth or a filler, the tiling comes to a dead end after placing one or two more polyominoes. So the only way to extend this tiling is to connect this meat outside the jaws to another meat by links. By repeatedly applying the same argument to the new meat, we extend the tiling to form the pattern illustrated in Figure \ref{fig_pattern}.
\end{itemize}

Finally, by the constraint of the length of the links, in order to tile the entire plane without overlaps or gaps with the pattern of Figure \ref{fig_pattern}, we must match a normal dent of one meat with a deeper dent with another meat by a link polyomino. According to our encoding method, this is equivalent to that the color of the two sides of the two simulated Wang tiles which are connected by the links must be the same. Therefore, the set of $8$ polyominoes can tile the plane if and only if the corresponding set of Wang tiles can tile the plane. This completes the proof.
\end{proof}

\section{Conclusion}\label{sec_con}

In this paper, we settle Ollinger's conjecture completely by showing that the $8$-polyomino tiling problem is undecidable. Two important techniques, namely a new orientation of the simulated Wang tiles, and a new encoding method, are applied to achieve the goal of decreasing the number of teeth and links at the same time. Going beyond Ollinger's conjecture, it is interesting to investigate whether the $k$-polyomino tiling problem is undecidable for some $k$ $(2 \leq k\leq 7)$.

\section*{Acknowledgements}
The first author was supported by the Research Fund of Guangdong University of Foreign Studies (Nos. 297-ZW200011 and 297-ZW230018), and the National Natural Science Foundation of China (No. 61976104).


\end{document}